\documentclass[11pt]{article}
\usepackage{amssymb,amsmath,amsfonts,amsthm}

\numberwithin{equation}{section}

\newtheorem{theorem}{Theorem}[section]
\newtheorem{lemma}{Lemma}[section]

\newtheorem{Proposition}[lemma]{Proposition}
\newtheorem{remark}[lemma]{Remark}
\newtheorem{example}[lemma]{Example}
\newtheorem{definition}[lemma]{Definition}

\def\1{\raisebox{2pt}{\rm{$\chi$}}}

\def\Carre#1#2{\vbox{
   \hrule height .#2pt
   \hbox{\vrule width .#2pt height #1pt \kern #1pt
      \vrule width .#2pt}
   \hrule height .#2pt}}

\def\grad{\nabla}

\def\inidat{u_0}

\def\Om{\Omega}
\def\e{\varepsilon}
\def\eps{\varepsilon}
\def\R{\mathbb{R}}

\def\Div{\textup{div}\,}

\renewcommand{\div}{\operatorname{div}}
\newcommand{\nau}{\nabla u}
\newcommand{\vnau}{\vert \nabla u \vert}

\newcommand{\scal}[2]{\left( #1\, ,\, #2 \right)}

\begin{document}

\title{Eventual regularity for the parabolic minimal surface equation}

\author{{G. Bellettini}\thanks{Dipartimento di Matematica, 
        Universit\`a di Roma Tor Vergata, 
        Via della Ricerca Scientifica 1, 00133 Roma, Italy,
        \newline e-mail: \texttt{belletti@mat.uniroma2.it}},
        {M. Novaga}\thanks{Dipartimento di Matematica, 
        Universit\`a di Pisa,
        Largo B. Pontecorvo 5, 56127 Pisa, Italy,
        \newline e-mail: \texttt{novaga@dm.unipi.it}},
        {G. Orlandi}\thanks{Dipartimento di Informatica, 
        Universit\`a di Verona,
        Strada le Grazie 15, 37134 Verona, Italy,
        \newline e-mail: \texttt{giandomenico.orlandi@univr.it}}
        }

\date{}
\maketitle

\begin{abstract}
\noindent We show that the parabolic minimal surface equation has an eventual
regularization effect, that is, the solution becomes smooth after a strictly positive finite
time.
\end{abstract}

\section{Introduction}\label{sec:int}

In this paper we consider the evolution problem
\begin{equation}\label{evol_curvatura0c}
\begin{cases}
\dfrac{\partial u}{\partial t} = {\rm div}
\left(\dfrac{\nabla u}{\sqrt{1+|\nabla u|^2}}\right)
&\qquad {\rm in} ~(0,\infty) \times \Om,
\\
\\
\dfrac{\nabla u}{\sqrt{1+|\nabla u|^2}}\cdot \nu^\Omega = 0 &\qquad
{\rm in} ~(0,\infty) \times \partial \Om,
\\
\\
u(0,\cdot) = u_0(\cdot) &\qquad {\rm in}~ \Om,
\end{cases}
\end{equation}
where $\Om\subset\R^N$ is a bounded open subset with smooth boundary,
 $\nu^\Om$ denotes the exterior unit normal to $\partial\Om$, and $\cdot$
is the scalar product in $\R^N$.

Problem \eqref{evol_curvatura0c} corresponds to the $L^2$-gradient flow
\cite{Brezis}
of the convex lower semicontinuous  functional
$F:L^2(\Om)\to [0,+\infty]$, defined as
\[
F(u) :=
\begin{cases}
\displaystyle
\int_\Om \sqrt{1+|\nabla u|^2}\,dx + |D^s u|(\Om) &\quad \text{if }u\in
BV(\Om),
\\
+\infty &\quad \text{otherwise},
\end{cases}
\]
where, for $u\in BV(\Om)$,  we write its distributional derivative $Du$ as
$$
Du=\nabla u\, dx + D^s u.
$$
Here, $D^s u$ is the singular part of the measure $Du$, with respect to the
Lebesgue measure (see \cite{Gi:84}), and $\vert D^s u\vert$ stands for its
total variation.

Equation \eqref{evol_curvatura0c} arises in several models of physical
systems
describing, for instance, the motion of capillary surfaces and
the motion of grain boundaries in annealing metals  (see \cite{Brakke}).
For such reasons, this evolution problem
has been already considered in the mathematical literature. In particular,
Lichnewski and Temam \cite{LT} showed existence of generalized solutions,
while Gerhardt \cite{gerhardt}  and Ecker \cite{ecker}
proved esimates on $u,\,u_t,\,|Du|$ similar to the ones we present
in this paper (see Lemma \ref{lem:crucial}).

We point out that equation
\eqref{evol_curvatura0c} should not be confused
with the mean curvature flow for graphs, which has been
deeply studied in \cite{EH1,EH2},  and reads as
\[
\frac{\partial u}{\partial t} = \sqrt{1+|\nabla u|^2}
\,{\rm div} \left(\dfrac{\nabla u}{\sqrt{1+|\nabla u|^2}}\right) .
\]

Let us state our main result.

\begin{theorem}\label{teo:main}
Let  $\inidat \in L^2(\Om)$ and let $u$ be the unique
solution of \eqref{evol_curvatura0c}.
Suppose  that one of the two following cases  hold:
\begin{itemize}
\item[1)] $~N=1$,
\item[2)] $~N>1$,  and the generalized graph of $\inidat$\footnote{Roughly, 
the  graph of $\inidat$ with the addition of  the ``vertical'' parts.},
$$
{\rm graph}(\inidat) := \overline{\partial \big\{(x,y): x\in \Om,\,y< u_0(x)\big\}\cap (\Om\times \R)},
$$
is a compact hypersurface of class $\mathcal C^{1,1}$, 
meeting orthogonally $(\partial\Om)\times\R$.
\end{itemize}
Then there exists $T > 0$ such that 
$$
u \in \mathcal
C^\omega((T,+\infty)\times\Om).
$$
Namely, $u(t)$ is analytic in $\Omega$, for any $t \in (T, +\infty)$. Moreover, 
$u(t)$ converges to the mean value $\overline u_0 := \frac{1}{\vert \Omega\vert} \int_\Omega \inidat~dx$ 
of $\inidat$ in 
$\mathcal C^\infty(\Omega)$ as $t \to +\infty$.
\end{theorem}

The assumption on ${\rm graph}(\inidat)$ when $N>1$ is technical,
but we are presently not able to remove it. Notice that
from this condition it follows that 
$$
u_0\in  L^\infty(\Om) \cap BV(\Om),
$$
and $u_0$ is of class $\mathcal C^{1,1}$ in a neighbourhood of $\partial\Om$,
with Neumann boundary condition on $\partial\Om$. 

Theorem \ref{teo:main} states that the solution $u(t)$ to
\eqref{evol_curvatura0c} becomes  smooth after some time $T$, which in
general is strictly positive.
This behaviour is somewhat different from the usual regularity results for 
parabolic partial differential equations.
Indeed, for this problem there is no instantaneous regularization of the
solution,  which holds for uniformly parabolic equations
but does {\it not} hold in general for \eqref{evol_curvatura0c}.

Another well-known degenerate parabolic problem  which shares
this property is the so-called
total variation flow, which has relevant applications in image analysis
and denoising (see, {\it e.g.}, 
 \cite{ROF,BeCaNo,ACMbook,TVHandbook,notes})\footnote{As shown in \cite{ACDM}, solutions to the total variation flow
in the whole of $\mathbb R^N$ 
become extinct (hence in particular smooth) in finite time,
even if the discontinuity set
does not disappear immediately
(see, {\it e.g.},  \cite{BeCaNo:05,CCN}).}.

\begin{example}\label{rem:esempio}\rm
As an example of eventual but not instantaneous regularization of $u$,
we consider the following situation:
$N=1$, $\Omega = (0,2)$, $c$ a positive constant, and
$$
u_0(x) := \begin{cases}
\sqrt{1-x^2} + c & {\rm if}~ x \in (0,1),
\\
- \sqrt{1 - (2-x)^2} & {\rm if}~ x \in (1,2).
\end{cases}
$$
Then $u_0$ is discontinuous at $x=1$, and
${\rm graph}(u_0)$ is a curve of class $\mathcal C^{1,1}$
consisting of two quarters of unit circles (hence with constant curvature
equal to $1$) and a vertical segment of length $c$, with the correct
boundary condition. Then
$$
u(t,x) =  \begin{cases}
-t + \sqrt{1-x^2} + c & {\rm if}~ (t,x) \in \left(0,\frac{c}{2}\right)\times (0,1),
\\
t - \sqrt{1 - (2-x)^2} & {\rm if}~ (t,x) \in \left(0,\frac{c}{2}\right)\times (1,2),
\end{cases}
$$
which is still discontinuous at $x=1$, because the
upper quarter of circle
as unit negative vertical velocity, while the
lower  quarter of circle
as unit positive  vertical velocity. Hence the time necessary
to let the jump disappear is $T:= c/2$,
and one checks that the solution becomes smooth in $(T, +\infty)\times
(0,2)$.
\end{example}

\smallskip

\noindent{\bf Acknowledgements.} This problem has been proposed to us by
our friend and collegue Vicent Caselles, who prematurely died in August 2013.
Without his contribution and insight this project would not have been possible. 
We are deeply indebted with him, and we dedicate this work to his memory.

\section{Notation and preliminary results}\label{sec:not}
We denote by $\partial F$ the subdifferential of $F$ in the sense of convex analysis, which defines a maximal
monotone operator in $L^2(\Om)$. A characterization of $\partial F$ is given in Remark \ref{remsub}.

\begin{definition}\label{def:strongsolution}
Let $u_0 \in L^2(\Omega)$.
We say that a function $u:[0,+\infty)\to \Om$ is 
a strong solution (briefly, a solution) of \eqref{evol_curvatura0c}, if
\begin{itemize}
\item[(i)]
$u \in H^1((0,T); L^2(\Om))\cap L^\infty((\tau,+\infty); BV(\Om))$ for any $T>0$ and $\tau >0$, 
\item[(iii)] the following inclusion holds:
\begin{equation}\label{eqweak}
u_t (t)+ \partial F(u(t))\ni 0 \qquad \text{for a.e. }t \in (0,+\infty),
\end{equation}
\item[(ii)] $\displaystyle \lim_{t\to 0^+} u(t) =u_0$ in $L^2(\Om)$.
\end{itemize}
\end{definition}

We let 
\[
X(\Om):= \big\{ z\in L^2(\Om;\R^N): \Div z\in L^2(\Om)\big\}.
\]
It is known that, for $z\in X(\Om)$,  the normal trace $[z,\nu^\Om]$ of $z$ on $\partial \Om$ 
is well defined (see \cite{Anz:83,ACMbook}).

\begin{remark}\label{remsub}\rm
Following \cite{ACMbook}, inclusion \eqref{eqweak} can be 
equivalently written as 
\begin{equation}\label{eqstrong}
\begin{cases}
u_t = \Div z &\qquad \text{in } \mathcal D'(\Omega), ~\text{for a.e. } t \in  (0,+\infty),
\\
z = \dfrac{\nabla u}{\sqrt{1+|\nabla u|^2}} 
&\qquad \text{a.e. in } (0,+\infty)\times\Om,
\\
z(t) \in X(\Om) &\qquad \text{for a.e.  } t\in  (0,+\infty),
\\
[z(t,\cdot),\nu^\Om] = 0 &\qquad \text{for a.e.  } t\in  (0,+\infty).
\end{cases}
\end{equation}
\end{remark}

Note that, in the expression of $z$,
only  the absolutely continuous 
part of the spatial gradient of $u$ is involved.

Let us recall the following results, proved in 
\cite[Theorems 3.2, 3.7, 3.11]{Brezis} (see also  \cite[Chapter 6]{ACMbook}).

\begin{theorem}\label{exunest}
Let $u_0 \in L^2(\Om)$. 
Then there exists a unique 
solution $u$ of \eqref{evol_curvatura0c}. Moreover:
\begin{itemize}
\item[(i)] the function $t \in (0,+\infty) \mapsto F(u(t))$ is nonincreasing, and
\[
\frac{d}{dt}F(u(t))=-\int_\Om u_t^2\,dx \qquad \text{for a.e. }t\in (0,+\infty);
\]
\item[(ii)] $\displaystyle \lim_{t\to +\infty} u(t) 
=\overline  u_0 := \frac{1}{|\Om|}\int_\Om u_0~dx$ in $L^2(\Om)$;
\item[(iii)] $\Vert u_t\Vert_{L^2(\Om)} 
\leq \dfrac{\Vert \inidat\Vert_{L^2(\Om)}}{t}$ for almost any $t >0$.
\end{itemize}
\end{theorem}

\begin{remark}\label{rem:regarding}\rm
Regarding point $(ii)$ of Theorem \ref{exunest}, from \cite[Theorem 3.11]{Brezis} it follows that $u(t)$ converges   in $L^2(\Om)$, as $t\to +\infty$, to 
a minimizer of $F$, that is, to a  constant in $\Om$. The value of this constant is fixed  from 
\[
\frac{d}{dt}\int_\Om u(t,x)\,dx = \int_\Om \Div z(t,x)\, dx = 0,
\]
where the last equality follows from the Gauss-Green Theorem \cite{Anz:83}. 
\end{remark}

The maximum and minimum principles ensure the following result.

\begin{Proposition}\label{pro:calanormalinfinito}
Let  $\inidat \in  L^\infty(\Om)$. Then the function
$$
t \in (0, +\infty) \to \Vert u(t) \Vert_{L^\infty(\Omega)}
$$ 
is nonincreasing.
\end{Proposition}

The following approximation result is proved in \cite[Theorem 3.16]{Brezis}.

\begin{Proposition}\label{pro:approx}
Let $\inidat,\,{\inidat}_n \in L^2(\Om)$ be 
such that 
$$
\lim_{n\to +\infty} \|\inidat -{\inidat}_n\|_{L^2(\Om)}=0\,.
$$
Let $u$ be the solution to 
\eqref{evol_curvatura0c},
and let $u_n$ be the solution to the first two equations of 
\eqref{evol_curvatura0c}, and with $u_n(0,\cdot) = u_{0n}(\cdot)$. 
Then, for all $T>0$ we have
\begin{equation}\label{eqapprox}
\lim_{n\to +\infty} \|u(t) -u_n(t)\|_{L^2(\Om)}=0 
\qquad \text{uniformly in }[0,T]\,.
\end{equation}
\end{Proposition}

\section{Proof of the main result}\label{sec:proof}
We start with the following estimates, 
which have been shown in \cite{gerhardt,ecker}. 
It can be useful to have a detailed proof, which we include here for completeness.

\begin{lemma}\label{lem:crucial}
Let $u_0 \in \mathcal C^\infty(\overline\Omega)$, 
and let $u$ be the solution to 
\eqref{evol_curvatura0c} given by Theorem \ref{exunest}. 
For all $t>0$ we have 
\begin{align}\label{eqnau}
\|\nau(t)\|_{L^\infty(\Omega)} &\leq\ \|\nabla u_0\|_{L^\infty(\Omega)}
\\\label{equt}
\| u_t(t)\|_{L^\infty(\Omega)} &\leq\ \|u_t(0)\|_{L^\infty(\Omega)}\,.
\end{align}
In particular, by parabolic regularity theory, from \eqref{eqnau} 
it follows that $u\in C^\infty([0,+\infty)\times\overline \Om)$.
\end{lemma}

\begin{proof}
Both estimates follow by a direct computation and by the maximum principle.
By \cite{Lu:95} we have that there exists $\tau >0$ such that 
$u \in \mathcal C^{(2+\alpha)/2,2+\alpha}([0,\tau)\times \Omega)$ for all $\alpha \in (0,1)$,
and therefore, by parabolic regularity, we have 
$u \in \mathcal C^\infty([0,\tau)\times \overline \Omega)$. 
 
Let us first show \eqref{eqnau}, arguing
at points in $(0,\tau) \times \Omega$.
 Differentiating \eqref{evol_curvatura0c}, we get
\begin{align*}
\frac{\partial}{\partial t} \left(\frac{\vnau^2}{2}\right) &= \nau \cdot \nau_t = \nau \cdot \nabla 
\left( \div\left( \frac{\nau}{\sqrt{1+\vnau^2}} \right) \right) 
\\
&=  \nau \cdot \nabla \left( \frac{\Delta u}{\sqrt{1+\vnau^2}} 
- \frac{\nabla \left(\frac{\vnau^2}{2}\right) \cdot \nabla u}{(1+\vnau^2)^{3/2}} \right) .
\end{align*}
Letting $\partial_i=\frac{\partial}{\partial x_i}$, 
$\partial_{ij}=\frac{\partial}{\partial x_i} \frac{\partial}{\partial x_j}$, and denoting by $\nabla^2 u$ the Hessian
of $u$, 
we compute
\begin{align*}
\nau \cdot \nabla \left( \frac{\Delta u}{\sqrt{1+\vnau^2}} \right)
=&\, \displaystyle \frac{\sum_{i=1}^N \partial_i u\, \Delta (\partial_i u)}{\sqrt{1+\vnau^2}} 
- \frac{\Delta u \left(\nau\cdot\nabla \left(\frac{\vnau^2}{2}\right)\right)}{(1+\vnau^2)^{3/2}}
\\
=&\, \frac{\Delta\left( \frac{|\nau|^2}{2}\right)-|\nabla^2 u|^2}{\sqrt{1+\vnau^2}} 
- \frac{\Delta u \left(\nau\cdot\nabla \left(\frac{\vnau^2}{2}\right)\right)}{(1+\vnau^2)^{3/2}},
\end{align*}
where $\vert \nabla^2 u\vert^2 := \sum_{i,j=1}^N (\partial_i \partial_j u)^2$. Also
\begin{align*}
\nau \cdot \nabla \left( \frac{\nabla \left(\frac{\vnau^2}{2}\right) \cdot \nabla u}{(1+\vnau^2)^{3/2}} \right) 
=&\, \frac{\scal{\frac{\nau}{\sqrt{1+|\nau|^2}}}{\nabla^2\left(\frac{|\nau|^2}{2} \right)\frac{\nau}{\sqrt{1+|\nau|^2}}}}
{\sqrt{1+|\nau|^2}}
\\ 
&+ \frac{\left|\nabla \left(\frac{\vnau^2}{2}\right)\right|^2}{(1+\vnau^2)^{3/2}} 
- 3\, \frac{
\left(\nabla \left(\frac{\vnau^2}{2}\right) \cdot \nabla u\right)^2}{(1+\vnau^2)^{5/2}}. 
\end{align*}
Summing up, we get
\begin{align}\nonumber
 \frac{\partial}{\partial t} \frac{\vnau^2}{2} \le &\, 
 \frac{\Delta \left(\frac{|\nau|^2}{2}\right)-\scal{\frac{\nau}{\sqrt{1+|\nau|^2}}}{\nabla^2\left(\frac{|\nau|^2}{2} \right)\frac{\nau}{\sqrt{1+|\nau|^2}}}}{\sqrt{1+|\nau|^2}} 
 \\ \label{eppo}
 & - \frac{\left \vert \nabla \left( \frac{\vnau^2}{2} \right) \right \vert^2}{(1+\vnau^2)^{3/2}} 
 + 3 \frac{ \left(\nabla\left(\frac{\vnau^2}{2}\right)\cdot \nabla u\right)^2}{(1+\vnau^2)^{5/2}}
\\ \nonumber & - \frac{\Delta u \left(\nau\cdot\nabla \left(\frac{\vnau^2}{2}\right)\right)}{(1+\vnau^2)^{3/2}}\,.
\end{align}
The estimate \eqref{eqnau} then follows from the maximum principle applied to \eqref{eppo}.
By standard arguments \cite{Lu:95}, it now  follows that 
$u \in \mathcal C^\infty([0,+\infty) \times \overline \Omega)$, and 
\eqref{eqnau} holds for all $t>0$.
\medskip

Let us now show \eqref{equt}. From \eqref{evol_curvatura0c}, we get
\begin{align}\label{eqzero} 
\frac{\partial}{\partial t} \left(\frac{u_t^2}{2}\right) =&\, u_t  \left( \div \left( \frac{\nau}{\sqrt{1+\vnau^2}} \right) \right)_t  
\\ \nonumber =&\, u_t \div\left( \frac{\nau_t}{\sqrt{1+\vnau^2}} \right)- u_t \div\left( \frac{(\nau_t \cdot \nau) \nau}{(1+\vnau^2)^{3/2}} \right). 
\end{align}
We have
\begin{align}\label{equno}
 u_t \div\left(\frac{\nabla u_t}{\sqrt{1+\vnau^2}}\right) &= u_t \sum_i \partial_i \left( \frac{\partial_i u_{t}}{\sqrt{1+\vnau^2}} \right) 
\\\nonumber &=  \frac{u_t \,\Delta  u_{t}}{\sqrt{1+\vnau^2}}-\sum_i \frac{ u_t \partial_i u_{t} \nabla u \cdot \nabla \partial_iu }{(1+\vnau^2)^{3/2}} 
\\ \nonumber
&= \frac{u_t \Delta u_t}{\sqrt{1+\vert \nau\vert^2}} - \frac{\nabla\left( \frac{u_t^2}{2} \right) \cdot \nabla \left(\frac{\vnau^2}{2} \right)}{(1+\vnau^2)^{3/2}} =: {\rm I} + {\rm II},
\end{align}
and
\begin{align}\nonumber
u_t \div\left( \frac{(\nau_t \cdot \nau)\nau}{(1+\vnau^2)^{3/2}} \right) 
=&  \,\Delta u \frac{\nabla u\cdot u_t \nabla u_t }{(1+\vnau^2)^{3/2}} 
+  \frac{u_t  \nau \cdot \nabla^2 u_t \nau}{(1+\vnau^2)^{3/2}} 
\\\nonumber
& + \frac{ \nabla^2 u \nau \cdot u_t \nabla u_t}{(1+\vnau^2)^{3/2}} 
\\\label{eqdue}
&- 3 \sum_i \partial_i u\frac{u_t (\nabla u_t\cdot \nabla u) (\nabla \partial_i u\cdot\nabla u)}{(1+\vnau^2)^{5/2}} 
\\\nonumber
=& \,\Delta u \frac{\nabla u \cdot\nabla (\frac{u_t^2}{2})}{(1+\vnau^2)^{3/2}} 
+ \frac{u_t \nau \cdot \nabla^2 u_t \nau}{(1+\vnau^2)^{3/2}}
\\\nonumber 
& + \frac{\nabla (\frac{\vnau^2}{2})\cdot\nabla (\frac{u_t^2}{2})}{(1+\vnau^2)^{3/2}} 
\\\nonumber
&-3\frac{ \left(\nabla (\frac{u_t^2}{2})\cdot\nabla u\right) \left(\nabla (\frac{\vnau^2}{2})\cdot\nau\right)}{(1+\vnau^2)^{5/2}} 
\\\nonumber
=:&\ {\rm III}
+ {\rm IV} + {\rm V} + {\rm VI}, 
\end{align}
with ${\rm V} = -{\rm II}$, hence 
\begin{equation}\label{sost}
\frac{\partial}{\partial t}\left(\frac{u_t^2}{2}\right) = 
{\rm I} - {\rm III} - {\rm IV} - 2 {\rm V} - {\rm VI}.
\end{equation}
Now, we write $u_t \Delta u = \Delta(u_t^2/2) - \vert \nabla u_t\vert^2$,  and 
$$u_t \grad u \cdot \grad^2 u_t \grad u = \grad u\cdot \grad^2 (u_t^2/2)\grad u
- (\grad u \cdot \grad u_t)^2,$$ 
whence
\begin{align*}
{\rm I} - {\rm IV} =&\ 
\frac{\Delta \left(\frac{u_t^2}{2}\right)-\scal{\frac{\nau}{\sqrt{1+|\nau|^2}}}{\nabla^2\left(\frac{u_t^2}{2} \right)\frac{\nau}{\sqrt{1+|\nau|^2}}}}{\sqrt{1+|\nau|^2}}
\\
&-\frac{|\nau_t|^2-\left(\frac{\nau}{\sqrt{1+|\nau|^2}}\cdot\nau_t\right)^2}{\sqrt{1+|\nau|^2}}.
\end{align*}
Substituting into \eqref{sost} we then get
\begin{align*}
\frac{\partial}{\partial t} \left(\frac{u_t^2}{2}\right) 
=& 
\ \frac{  \Delta \left( \frac{u_t^2}{2} \right) 
-
           \frac{ \nau }{ \sqrt{1+|\nau|^2} }  \cdot  \nabla^2\left(\frac{u_t^2}{2} \right) \frac{\nau}{\sqrt{1+|\nau|^2}}
        }{
           \sqrt{1+|\nau|^2}
          }   
\\  \nonumber 
&
   - \frac{|\nau_t|^2-\left(\frac{\nau}{\sqrt{1+|\nau|^2}}\cdot\nau_t\right)^2}{\sqrt{1+|\nau|^2}} 
- 
\Delta u \frac{\nabla u \cdot\nabla (\frac{u_t^2}{2})}{(1+\vnau^2)^{3/2}}
\\ 
&
+3 \frac{\left(\nabla(\frac{u_t^2}{2}) \cdot \nabla u\right) 
\left(\nabla(\frac{\vnau^2}{2}) \cdot \nau\right)}{(1+\vnau^2)^{5/2}} 
- 
2\frac{\nabla\left( \frac{u_t^2}{2} \right) \cdot \nabla \left(\frac{\vnau^2}{2}\right)}{(1+\vnau^2)^{3/2}}.
\end{align*}
Observing that $\vert \grad u_t\vert^2 \leq \left(\frac{\grad u}{\sqrt{1 + \vert \grad u\vert^2}} \cdot \grad u_t \right)^2$, we get
\begin{align}\nonumber
\frac{\partial}{\partial t} \left(\frac{u_t^2}{2}\right) \leq 
&
\frac{  \Delta \left( \frac{u_t^2}{2} \right) 
-
           \frac{ \nau }{ \sqrt{1+|\nau|^2} }  \cdot  \nabla^2\left(\frac{u_t^2}{2} \right) \frac{\nau}{\sqrt{1+|\nau|^2}}
        }{
           \sqrt{1+|\nau|^2}
          }   
\\\label{eppa}
&
- 
\Delta u \frac{\nabla u \cdot\nabla (\frac{u_t^2}{2})}{(1+\vnau^2)^{3/2}}
\\\nonumber 
&
+3 \frac{\left(\nabla(\frac{u_t^2}{2}) \cdot \nabla u\right) 
\left(\nabla\left(\frac{\vnau^2}{2}\right) \cdot \nau\right)}{(1+\vnau^2)^{5/2}} 
\\\nonumber
&
- 
2\frac{\nabla\left( \frac{u_t^2}{2} \right) \cdot \nabla \left(\frac{\vnau^2}{2}\right)}{(1+\vnau^2)^{3/2}}.
\end{align}
The estimate \eqref{equt} follows as above from the maximum principle applied to \eqref{eppa}.
\end{proof}

\smallskip

\noindent{\em Proof of Theorem \ref{teo:main}.}

{\bf Case 1.} Assume that $N=1$. Note that $X(\Omega)=  H^1(\Omega)$.
Let $L\subset (0,+\infty)$ be a 
set of zero Lebesgue measure such that the partial
differential  equation in \eqref{eqstrong} and Theorem \ref{exunest} $(iii)$
hold for any $t\in J:= (0,+\infty)\setminus L$. 
We deduce that 
for all $\tau>0$ there exists a constant $C=C(\tau)>0$ such that 
$$
\|z(t,\cdot)\|_{H^1(\Om)}\le C, \qquad  t\ge\tau, ~ t\in J.
$$
In particular, 
by Sobolev embedding,
$z(t,\cdot)$ is $1/2$--H\"older continuous, uniformly for $t\ge \tau$, $t\in J$.
Therefore, for all $\eta\in (0,1)$ there exists $\e=\e(\eta)>0$ 
independent of $t \geq \tau$, $t \in J$,  such that, if 
\begin{equation}
 |z(t,x_0)|\ge \eta \qquad \text{for~some}~ (t,x_0) \in J\times\Omega ,
 \end{equation}
 then 
$$
|z(t,x)|\ge \frac\eta 2 \qquad  \text{for all }x\in (x_0-\e,x_0+\e)
\cap \Om. 
$$
Hence
\begin{equation}
\text{either}\quad u_x(t,x)\ge \frac{\eta}{\sqrt{4-\eta^2}} \quad
\text{or}\quad u_x(t,x)\le -\frac{\eta}{\sqrt{4-\eta^2}}\,,
\end{equation}
for almost  every  $x\in (x_0-\e,x_0+\e)\cap\Om$.
We then get
\begin{eqnarray}\nonumber
\Vert u(t)- \overline u_0\Vert^2_{L^2(\Omega)} &\geq& 
\min_{\alpha \in \R} \Vert u(t)- \alpha\Vert^2_{L^2(\Omega)}
\\
&\geq& 
\min_{\beta \in \R}  f(\beta),
\end{eqnarray}
where\footnote{
Indeed, if $c:= \frac{\eta}{4-\eta^2}$, 
for any $ \alpha\in\R$ and $ x\in (x_0-\eps,x_0+\eps) \cap \Omega$
we have
$|u(t,x)-\alpha| = \vert u(t,x) - \alpha + \int_{x_0}^x u'(t,\xi)~d\xi \vert 
 \ge |u(t,x_0)-\alpha+c(x-x_0)|=|c x-\beta|$,
where $\beta:=\alpha-u(t,x_0)+c x_0$. 
Hence
$\int_\Om (u(t,x)-\alpha)^2~dx\ge \int_\Om (cx-\beta)^2~dx$, and therefore
$\min_{\alpha \in \R}\int_\Om (u(t,x)-\alpha)^2\ge \min_{\beta\in \R} \int_\Om (cx-\beta)^2~dx$.
}
$$
f(\beta):=  \int_0^\e 
\left( \frac{\eta}{\sqrt{4-\eta^2}}\,x -\beta\right)^2\,dx, \qquad \beta \in \R,
$$
and we use
that $|\Om\cap (x_0-\eps,x_0+\eps)|\ge \eps$ 
as soon as  $\eps \in (0, \vert\Om\vert]$.

One checks that 
$\min_{\beta \in \R}  f(\beta) = 
f
 \left(
  \frac{\eps \eta}{2 \sqrt{4-\eta^2}}
 \right)
=\frac{\eta^2}{4-\eta^2}\, \frac{\eps^3}{12}$. Hence
$$
\Vert u(t)- \overline u_0\Vert^2_{L^2(\Omega)}
\geq  \frac{\eta^2}{4-\eta^2}\, \frac{\eps^3}{12}\,.
$$
Since $u(t)\to \overline u_0$ in $L^2(\Om)$ as $t \to +\infty$ by Theorem \ref{exunest} $(ii)$, 
in order not to have a contradiction 
it follows that, 
given $\eta \in (0,1)$, 
there exists $T=T(\eta)$ such that
$\vert z(t,x)\vert<  \eta$ for all $(t,x)\in J\times\Om$, $t\ge T$. 
Therefore
$u(t,\cdot)$ is $(\eta/\sqrt{4-\eta^2})$-Lipschitz in $\Om$ for all $t>T$, $t \in J$, and hence 
for all $t > T$. 
Thus, by parabolic regularity theory, 
$$
u(t,\cdot)\in \mathcal C^\omega(\Om)\qquad \text{for all }t>T\,, 
$$
and 
$u(t)\to \overline u_0$ in $\mathcal C^\infty(\Om)$ as $t\to +\infty$.

\smallskip

{\bf Case 2.} Assume that $N>1$. Let  $u_0\in   L^\infty(\Om)\cap BV(\Om)$, and
suppose that ${\rm graph}(\inidat)$ is a hypersurface 
of class $\mathcal C^{1,1}$ meeting orthogonally $\partial\Om$. 
We divide the proof into four steps.

{\it Step 1.} There exists a sequence 
$(u^n_0) \subset {\mathcal C}^\infty(\Omega) \cap {\rm Lip}(\Om)$ 
converging to $u_0$ in $L^2(\Omega)$ and such that
\begin{equation}\label{stimatn}
\sup_{n \in \mathbb N}  
\left
\Vert 
{\rm div}\left(\frac{\nabla u^n_0}{\sqrt{1+|\nabla u^n_0|^2}}\right)
\right\Vert_{L^\infty(\Omega)}  < +\infty.
\end{equation}

Indeed, let $\Sigma(t)$ be the mean curvature evolution starting from 
${\rm graph}(\inidat) = \Sigma(0)$, with Neumann boundary conditions on 
$\partial\Om\times\R$ (see \cite{Stahl,Buckland}). Since ${\rm graph}(\inidat)$ is of class 
$\mathcal C^{1,1}$, there exists an evolution $t\mapsto \Sigma(t)$, with 
$t\in [0,\tau]$ for some $\tau>0$, such that 
$\Sigma(t)$ is of class $\mathcal C^\infty$ (actually analytic) and $\Sigma(t)$ is
of class $\mathcal C^{1,1}$\cite{Stahl}\footnote{The fact that $\Sigma(t)$ is of class
$\mathcal C^{1,1}$ uniformly in $[0,\tau]$ follows from the assumption on ${\rm graph}(u_0)$ and the estimates 
in \cite{Stahl} (see, {\it e.g.},   \cite[Ch. 13]{Bel} for related references and a precise argument).} uniformly in $[0,\tau]$.
 Moreover, 
letting $\nu(t)$ be the unit normal to $\Sigma(t)$ pointing upward,
so that $\nu_{n+1}(0)\ge 0$ on $\Sigma(0)$,
by the strong maximum principle we have $\nu_{n+1}(t)> 0$ on $\Sigma(t)$ 
for any $t\in (0,\tau]$. That is, $\Sigma(t)$ is the graph of a function 
$v(t)\in {\mathcal C}^\infty(\Omega) \cap {\rm Lip}(\Om)$.
We conclude by letting 
$$
u^n_0:=v(t_n),
$$
 where $(t_n)\subset (0,\tau)$ 
is a sequence converging to $0$ as $n \to +\infty$, so that ${\rm graph}(u^n_0)$ is of class 
$\mathcal C^{1,1}$ uniformly in $n\in \mathbb N$, which implies \eqref{stimatn}.

{\it Step 2.} The function $u$ satisfies 
\begin{equation}\label{stut}
\sup_{t\in (0,+\infty)} \Vert u_t (t)\Vert_{L^\infty(\Omega)} < +\infty.
\end{equation}
Let $u_n(t)$, $t\in [0,+\infty)$, be the solution 
of the first two equations of \eqref{evol_curvatura0c},  with initial condition  $u_n(0) = u^n_0$,
where $u^n_0$ is as in {\it step 1}. {}From Lemma \ref{lem:crucial} we have $u_n
\in \mathcal C^\infty([0,+\infty) \times \overline \Omega)$; 
{}from \eqref{stimatn} and \eqref{equt}, 
it follows that 
\begin{equation}\label{stimatons}
\sup_{n \in \mathbb N} 
\Vert \partial_t u_n(t)\Vert_{L^\infty(\Omega)} \le C
\qquad \forall t\in (0,+\infty),
\end{equation}
where $C>0$ is a constant bounding the left hand side of  \eqref{stimatn}.
Notice that \eqref{stimatons} is equivalent to say that the functions $u_n(\cdot,x)$
are $C$-Lipschitz on $(0,+\infty)$, uniformly in $n\in\mathbb N$ and $x\in\Om$.
Recalling that, by Proposition \ref{pro:approx}, $\lim_{n \to +\infty}u_n = u$ in $L^2((0,T)\times\Om)$ for all $T>0$, we can
extract a (not relabelled) subsequence so that 
$\lim_{n \to +\infty}u_n = u$  almost everywhere in $(0,+\infty)\times\Om$.
Passing to the limit, as $n \to +\infty$,  in the inequality
$|u_n(t,x)-u_n(s,x)|\le C|t-s|$  for almost every $(t,s,x) \in (0,+\infty)\times (0,+\infty)\times \Omega$, 
we get that $u(\cdot,x)$ is also $C$-Lipschitz in $(0,+\infty)$,
uniformly in $x\in\Om$, which gives  \eqref{stut}.

{\it Step 3.} We have that\footnote{$\alpha$ in general cannot be taken equal to one, see \cite{Amb97}.} 
\begin{equation}\label{strip}
{\rm graph}(u(t)) \text{ is }\mathcal C^{1,\alpha} 
\text{ for any }\alpha\in (0,1),
\text{ uniformly in }t\in (0,+\infty).
\end{equation} 

To prove assertion \eqref{strip} it is enough to 
closely follow  \cite[Proposition 4.4]{CN} and use
 \eqref{stut}:
we repeat here the argument for completeness. 
Suppose  by contradiction that \eqref{strip} does not hold. Then, for any $n \in \mathbb N$,  we can find 
$(t_n, x_n,y_n)\in (0,+\infty) \times \Om\times \R$  such that
$(x_n,y_n)\in {\rm graph}(u(t_n))$
and, for all $\rho>0$, the hypersurfaces ${\rm graph}(u(t_n))\cap B_\rho(x_n,y_n)$\footnote{Here $B_\rho(x,y)$ 
denotes the ball of radius $\rho$ in $\R^{N+1}$
centered at $(x,y)\in \overline\Om\times\R$.} 
are not uniformly $\mathcal{C}^{1,\alpha}$ with respect to $n$.  
Letting $\widetilde u_n(x) := u(t_n, x)-u(t_n, x_n)$, from \eqref{evol_curvatura0c} and \eqref{stut} we have that 
\begin{equation}\label{curvtilde}  
-\text{div}\left(\frac{ \nabla\widetilde u_n(x) }{\sqrt{1+|\nabla \widetilde u_n(x)|^2}}\right)= \kappa_n(x),    \qquad
x \in \Omega.
\end{equation} 
with 
\begin{equation}\label{eq:unifh}
\sup_{n \in \mathbb N}\|\kappa_n\|_{L^\infty(\Omega)}< +\infty.
\end{equation}
In particular, $\widetilde u_n$ is a minimizer of
the prescribed curvature functional
\[
 v\in BV(\Om) \mapsto \int_\Om \left(\sqrt{1+|\nabla v|^2} - \kappa_n v\right)\,dx + |D^s v|(\Om).
\]

{}From \eqref{eq:unifh} and the compactness theorem for quasi minimizers of the perimeter \cite{Amb97}, \cite{Gi:84}, 
the hypersurfaces
${\rm graph}(\widetilde u_n)$
converge in  $L^1(\Omega \times \R)$, and 
 up to a (not relabelled) subsequence, to a limit hypersurface $\Gamma_\infty\subset\overline\Om\times \R$ 
of class $\mathcal C^{1,\alpha}$, for all $\alpha\in (0,1)$.
Possibly passing to a further subsequence, 
we can also assume that $\lim_{n \to +\infty}x_n = x_\infty$, for some $x_\infty \in \overline\Om$. 
Observe that $\widetilde u_n(x_n) = 0$, and so $(x_\infty, 0) \in \Gamma_\infty$.

By \cite[Theorem 1]{Tamanini} there exists $\rho>0$ such that both  ${\rm graph}(\widetilde u_n)\cap B_\rho(x_\infty,0)$ and 
$\Gamma_{\infty}\cap B_\rho(x_\infty,0)$ can be written as graphs of functions of $N$ variables,
in the normal direction to $\Gamma_\infty$ at $(x_\infty,0)$.
Therefore, by regularity of minimizers of the prescribed curvature functional \cite{m},
the hypersurfaces ${\rm graph}(\widetilde u_n)\cap B_\rho(x_\infty,0)$ are uniformly (with respect to $n\in \mathbb N$) of class $\mathcal C^{1,\alpha}$ for all $\alpha\in (0,1)$,
thus leading to a contradiction.  

{\it Step 4.} From \eqref{strip} if follows that the vector field $z(t,\cdot)$ 
in \eqref{eqstrong} is $\alpha$-H\"older continuous in $\Om$, uniformly with 
respect to $t\in (0,+\infty)$. We can now proceed as in case 1, with only minor changes. Indeed, letting $L$ and $J\subset:=  (0,+\infty)\setminus L$ 
be as in case 1, from the 
H\"older continuity of $z$ we get that,
for all $\eta\in (0,1)$, there exists $\e=\e(\eta)>0$ independent of $t \in J$,  such that,
if for some $\nu \in \mathbb S^{N-1}$, 
\begin{equation}
 z(t,x_0)\cdot\nu\ge \eta \qquad \text{for~some}~ (t,x_0) \in J\times\Omega, 
 \end{equation}
 then 
$$
z(t,x)\cdot\nu\ge \frac\eta 2 \qquad  \text{for all }x\in B_\e(x_0)\cap \Om. 
$$
It then follows 
\begin{equation}
\nabla u(t,x)\cdot\nu\ge \frac{\eta}{\sqrt{4-\eta^2}} 
\qquad \text{for almost  every }x\in B_\e(x_0)\cap \Om,
\end{equation}
which implies, as in case 1, 
\begin{eqnarray}\nonumber
\Vert u(t)- \overline u_0\Vert^2_{L^2(\Omega)} &\geq& 
\min_{\alpha \in \R} \Vert u(t)- \alpha\Vert^2_{L^2(\Omega)}
\\
&\geq& 
\min_{\beta \in \R} 
g(\beta),
\end{eqnarray}
where 
$$
g(\beta) := \int_{B_\e(x_0)\cap \Om} 
\left( \frac{\eta}{\sqrt{4-\eta^2}}\,x\cdot\nu -\beta\right)^2\,dx, \qquad
\beta \in \R.
$$
One checks that 
$$
\begin{aligned}
\min_{\beta \in \R} 
g(\beta) 
= &\ g\left(
\frac{\eta \int_{B_\eps(x_0)\cap \Omega} x\cdot \nu~dx}{\sqrt{4-\eta^2} \vert B_\eps(x_0) \cap \Omega\vert}
         \right)
\\
= & \ \frac{\eta^2}{4-\eta^2} 
 \left( 
         \int_{ B_\eps(x_0)\cap \Omega } (x\cdot \nu)^2~dx
         -
            \left(
            \int_{ B_\eps(x_0)\cap \Omega} x\cdot \nu~dx
            \right)^2
\right) 
\\ 
\geq & \ C\,\frac{\eta^2}{4-\eta^2}\,\e^{N+2},
\end{aligned}
$$
for all $\eps \in (0,\eps_0(\Om))$, where $\eps_0(\Omega)>0$ depends only on $\Omega$, and  the constant $C>0$ depends only on the dimension $N$. 
Hence
$$
\Vert u(t)- \overline u_0\Vert^2_{L^2(\Omega)} \geq 
C\,\frac{\eta^2}{4-\eta^2}\,\e^{N+2}.
$$

Since $\Vert u(t)- \overline u_0\Vert^2_{L^2(\Omega)}\to 0$ as $t \to +\infty$ by Theorem \ref{exunest}, 
it follows that, given $\eta \in (0,1)$, 
there exists $T=T(\eta)$ (in particular independent of $\nu$) such that 
$\vert z(t,x)\vert\leq   \eta$ for all $(t,x)\in J\times\Om$, $t\ge T$. 
Therefore
$u(t,\cdot)$ is $(\eta/\sqrt{4-\eta^2})$-Lipschitz in $\Om$ for all $t>T$, so that
by parabolic regularity theory
$u(t,\cdot)\in \mathcal C^\omega(\Om)$ for all $t > T$, 
and 
$u(t)\to \overline u_0$ in $\mathcal C^\infty(\Om)$ as $t\to +\infty$.
\qed

\smallskip
We conclude the paper with an example showing that,
in contrast with  the one-dimensional case,
in higher dimensions there is no instantaneous regularization
of ${\rm graph}(u(t))$\footnote{In dimension one, we have proven that 
$z(t,\cdot)$ becomes instantaneously $1/2$-H\"older continuous, and this implies
that the graph of $u(t,\cdot)$ becomes instantaneously 
of class $\mathcal C^{1,1/2}$.}.

\begin{example}\rm
Let $N \geq 3$, $\Omega = B_1$ be the unit ball of $\R^N$ centered at the origin, and 
$\sigma := \frac{1}{N-1}$. 
Let $u_0(x) :=  1/|x|$, $x \in \Omega \setminus \{0\}$; notice that  $u_0\in L^2(\Om)\cap  BV(\Om)$.
Let $u$  be the solution to
\eqref{evol_curvatura0c} given by Theorem \ref{exunest}, so that
 $u(t,\cdot)\in L^2(\Om)\cap  BV(\Om)$ for all $t>0$.
Let us check that the function
$$
v(t,x):= \frac{a(t)}{|x|}, \qquad t>0, ~x \in \Omega \setminus \{0\}
$$
where 
$$
a(t):= \max\big(1-(N-1)\,t,0\big), \quad t >0, 
$$ 
is a subsolution
of \eqref{eqstrong}. For all $(t,x)\in \left(0, \sigma\right)\times(\Om \setminus \{0\})$, a direct computation gives
$\vert \grad v(t,x)\vert^2 = \frac{a(t)^2}{\vert x\vert^4}$, and 
$$
\begin{aligned}
& {\rm div}\left( \frac{\nabla v(t,x)}{\sqrt{1+|\nabla v(t,x)|^2}}\right) 
\\
= & 
-a 
\left(
\frac{N-3}{\vert x\vert} \frac{1}{\sqrt{a(t)^2 + \vert x\vert^4}}
+
 \frac{2}{\vert x\vert} \frac{a(t)^2}{(a(t)^2 + \vert x\vert^4 )^{3/2}}
\right) 
\\
& \geq 
-a 
\frac{N-1}{\vert x\vert} \frac{1}{\sqrt{a(t)^2 + \vert x\vert^4}}.
\end{aligned}
$$
Hence
\begin{eqnarray*}
v_t(t,x) \ =\ - \frac{(N-1)}{|x|} 
&\le& - \frac{(N-1)}{|x|} \,\frac{a(t)}{\sqrt{a(t)^2+|x|^4}}
\\
&\le& {\rm div}\left( \frac{\nabla v(t,x)}{\sqrt{1+|\nabla v(t,x)|^2}}\right)
\end{eqnarray*}
and therefore $v$ is a subsolution
of \eqref{eqstrong}.
By comparison principle (see \cite{ACMbook}) it follows 
that $u \ge v $ almost everywhere in  $(0,+\infty)\times \Om$. As a consequence,
${\rm graph}(u(t))$ is not of class $\mathcal C^{1}(\Omega)$ for $t\in [0,\sigma)$.
\end{example}


\end{document}